\documentclass[a4paper,10pt]{article}

\usepackage[]{amsmath,amssymb}
\usepackage[]{amsthm,enumerate}
\usepackage{verbatim,bm}

\newtheorem{theorem}{Theorem}[section]
\newtheorem{proposition}[theorem]{Proposition}
\newtheorem{lemma}[theorem]{Lemma}
\newtheorem{corollary}[theorem]{Corollary}

\theoremstyle{definition}
\newtheorem{definition}[theorem]{Definition}
\newtheorem{example}[theorem]{Example}

\newtheorem{problem}[theorem]{Problem}

\theoremstyle{remark}
\newtheorem{remark}[theorem]{Remark}

\numberwithin{equation}{section}

\begin{document}

\title{Unbiased orthogonal designs}
\author{
 Hadi Kharaghani\thanks{Department of Mathematics and Computer Science, University of Lethbridge,
Lethbridge, Alberta, T1K 3M4, Canada. \texttt{kharaghani@uleth.ca}} 
\and  
 Sho Suda\thanks{Department of Mathematics Education,  Aichi University of Education, 1 Hirosawa, Igaya-cho, Kariya, Aichi 448-8542, Japan. \texttt{suda@auecc.aichi-edu.ac.jp}}
}

\maketitle
\begin{abstract} 
The notion of unbiased orthogonal designs is introduced as a generalization among unbiased Hadamard matrices, unbiased weighing matrices and quasi-unbiased weighing matrices. 
We provide upper bounds and several constructions for mutually unbiased orthogonal designs.  
As an application, mutually quasi-unbiased weighing matrices for various parameters are obtained.   
\end{abstract}

\section{Introduction}
A {\em  Hadamard matrix of order $n$}  is an $n\times n$ $(1,-1)$-matrix $H$ such that $H H^\top=n I_n$, where $H^\top$ denotes the transpose of $H$ and $I_n$ denotes the identity matrix of order $n$.
A {\em weighing matrix of order $n$ and weight $k$} is an $n\times n$ $(0,1,-1)$-matrix $W$ such that $W W^\top=k I_n$. 
Recently unbiased Hadamard matrices and unbiased weighing matrices have been studied \cite{BKR,hko,LMO}.
Two Hadamard matrices $H$ and $K$ of order $n$ are called {\it unbiased}
if $HK^\top={\sqrt n} L$ for some Hadamard matrix $L$.
Two weighing matrices $H$ and $K$ of order $n$ and weight $k$ are called {\it unbiased}
if $HK^\top={\sqrt k} L$ for some weighing matrix $L$ of weight $k$ \cite{BKR,hko}. 
Mutually unbiased weighing matrices of weight $4$ naturally arise in the minimum vectors of root lattices admitting a decomposition of disjoint  orthogonal bases \cite[Theorem 3.5]{NS}. 
In the paper \cite{BKR}, Best, Kharaghani and Ramp posed the question for a construction of $2^{2t+1}$ Hadamard matrices $H_1,\ldots,H_{2^{2t+1}}$ of order $2^{2t+1}$ such that the entries of $H_i H_j^\top$ are $0,\pm2^{t+1}$ for any distinct $i,j$. 
In order to answer their question and consider more general situations, the concept of {\it quasi-unbiased weighing matrices} was given in \cite{NS}, see 
Section~\ref{sec:pre} for the definition. 
An answer was obtained by considering the BCH codes with cosets by the first order Reed-Muller code \cite[Theorem 4.4]{NS}. 
Both objects are related to a spread in a partial geometry \cite{B} or in a strongly regular graph \cite{HT} and yield a symmetric association scheme \cite{S}.

A generalized concept for a Hadamard matrix and a weighing matrix is an orthogonal design, see 
Section~\ref{sec:pre} for the definition. 
In this paper, as a unifying way to study unbiased Hadamard matrices, unbiased weighing matrices and quasi-unbiased weighing matrices, the concept of unbiased orthogonal designs is introduced.
Connecting  unbiased orthogonal designs with unbiased weighing matrices, we obtain the upper bound for the number of mutually unbiased orthogonal designs. 
We provide various constructions of unbiased orthogonal designs to use direct sum and tensor product for matrices, mutually suitable Latin squares.

The main result of the paper is Theorem~\ref{thm:constuod}, which constructs mutually unbiased orthogonal designs from a weighing matrix and an orthogonal design. 
The significance of the construction provided here is that  any weighing matrix and any orthogonal design of the same order can be used to construct unbiased orthogonal designs.
Furthermore we demonstrate how the plug-in method provides mutually quasi-unbiased weighing matrices from unbiased orthogonal designs with Goethals-Seidel matrices and Williamson type matrices. 

The organization of the paper is as follows. 
In Section~\ref{sec:pre} we prepare some notations ans results needed later. 
In Section~\ref{sec:uod} we introduce the concept of unbiased orthogonal designs, and extend constructions for unbiased Hadamard matrices and related topics to unbiased orthogonal designs. 
In Section~\ref{sec:app} we provide a new construction of quasi-unbiased weighing matrices from a finite ring with unity. 
By use of this construction, we obtain mutually unbiased orthogonal designs.
Applications  are also provided.  
In Section~\ref{sec:ff}, we investigate the properties for  some mutually quasi-unbiased weighing matrices constructed from Theorem~\ref{thm:constuod}, and finally we discuss unbiasedness for unit orthogonal designs in Section~\ref{sec:uuod}.

\section{Preliminaries}\label{sec:pre}
In this section, we present  notations and results to be used throughout the paper. 
\begin{definition}
An {\em orthogonal design of order $n$ and type $(s_1,\ldots,s_u)$ in variables $x_1,\ldots,x_u$} is a $(0,\pm x_1,\ldots,\pm x_u)$-matrix $D$, where $x_1,\ldots,x_u$ are distinct commuting indeterminates, such that $D D^\top=(s_1x_1^2+\cdots+s_u x_u^2)I_n$.
We denote it by {\it $OD(n;s_1,\ldots,s_u)$}. 
\end{definition}
Letting $(0,1,-1)$-matrices $W_1,\ldots,W_u$ be such as $D=\sum_{i=1}^u x_i W_i$,  
it holds that $W_i$ is a weighing matrix of order $n$ and weight $s_i$ for any $i$. 

We recall the existence of orthogonal designs of order $2^t$, $t$ a positive integer.  
There exist orthogonal designs $D$ of order $2,4,8$ and type $(1,1),(1,1,1,1),(1,1,1,1,\\ 1,1,1,1)$ respectively as follows:
\begin{align*}
D=\begin{pmatrix} 
x_1&x_2\\
-x_2&x_1
\end{pmatrix},
D=\begin{pmatrix} 
x_1&x_2&x_3&x_4\\
-x_2&x_1&x_4&-x_3\\
-x_3&-x_4&x_1&x_2\\
-x_4&x_3&-x_2&x_1
\end{pmatrix},
\end{align*}
\begin{align*} 
D=\begin{pmatrix} 
x_1&x_2&x_3&x_4&x_5&x_6&x_7&x_8\\
-x_2&x_1&x_4&-x_3&x_6&-x_5&x_8&-x_7\\
-x_3&-x_4&x_1&x_2&-x_7&x_8&x_5&-x_6\\
-x_4&x_3&-x_2&x_1&x_8&x_7&-x_6&-x_5\\
-x_5&-x_6&x_7&-x_8&x_1&x_2&-x_3&x_4\\
-x_6&x_5&-x_8&-x_7&-x_2&x_1&x_4&x_3\\
-x_7&-x_8&-x_5&x_6&x_3&-x_4&x_1&x_2\\
-x_8&x_7&x_6&x_5&-x_4&-x_3&-x_2&x_1
\end{pmatrix}.
\end{align*}
For $t>3$, there exists an orthogonal design  $D$ of order $2^t$ and type $(s_i)_{i=1}^{2t}=(1,1,1,1,2,2,4,4,\ldots,2^{t-2},2^{t-2})$ \cite{R}. 
Thus we have: 
\begin{lemma}\label{lem:odw}
\begin{enumerate}
\item For any positive integer $t$, there exists an orthogonal design of order $2^t$ and type $(s_1,\ldots,s_{u})$ where 
\begin{align*}
u=2,(s_i)_{i=1}^{2}&=(1,1) \text{ if } t=1,\\
u=4,(s_i)_{i=1}^{4}&=(1,1,1,1) \text{ if } t=2,\\
u=8,(s_i)_{i=1}^{8}&=(1,1,1,1,1,1,1,1) \text{ if } t=3,\\
u=2t,(s_i)_{i=1}^{2t}&=(1,1,1,1,2,2,4,4,\ldots,2^{t-2},2^{t-2}) \text{ if } t>3.  
\end{align*}
\item For any positive integers $t,k$ such that $k\leq 2^t$, there exists a weighing matrix of order $2^t$ and weight $k$. 
\end{enumerate}
\end{lemma}
\begin{proof}
(1) is already seen. 
It holds
\begin{align}\label{eq:odw}
\textstyle \left\{\sum_{i\in S}s_i\mid \emptyset\neq S\subset \{1,\ldots,u\}\right\}=\{k\in\mathbb{Z}\mid1\leq k\leq 2^{t}\},  
\end{align}
from which we obtain (2) by substituting $1,0$ into suitable variables.
\end{proof}

The concept of quasi-unbiased weighing matrices was introduced in \cite{NS} as a generalization of unbiased weighing matrices \cite{BKR,hko}.  
\begin{definition}
Two weighing matrices $W_1,W_2$ of order $n$ and weight $k$ are said to be {\em quasi-unbiased for  parameters $(n,k,l,a)$} if 
$(1/\sqrt{a})W_1 W_2^\top$ is a weighing matrix of order $n$ and weight $l$. 
 Weighing matrices $W_1,\ldots,W_f$ of order $n$ and weight $k$ are {\em mutually quasi-unbiased for parameters $(n,k,l,a)$} if any distinct two of them are quasi-unbiased for the parameters.
\end{definition}
If there exist quasi-unbiased weighing matrices for parameters $(n,k,l,a)$, then it holds that $l=k^2/a$.
Using this equality, it is easily shown that for weighing matrices $W_1,W_2$ of order $n$ and weight $k$,  $W_1,W_2$ are quasi-unbiased for parameters $(n,k,l,a)$ if and only if $(1/\sqrt{a})W_1W_2^\top$ is a $(0,1,-1)$-matrix. 
The case for the parameters $(n,n,l,a)$ was studied in \cite{AHS} from the viewpoint of coding theory.

A \emph{(symmetric) association scheme of class $d$}
with vertex set $X$ of size $n$ 
is a set of non-zero $(0,1)$-matrices $A_0, \ldots, A_d$, which are called {\em adjacency matrices}, with
rows and columns indexed by $X$, such that:
\begin{enumerate}
\item $A_0=I_n$.
\item $\sum_{i=0}^d A_i = J_n$, $J_n$ is the all-one matrix of order $n$.
\item For any $i\in\{0,1,\ldots,d\}$, $A_i^\top=A_i$.
\item For any $i,j\in\{0,1,\ldots,d\}$, $A_iA_j=\sum_{k=0}^d p_{ij}^k A_k$
for some $p_{ij}^k$'s.
\end{enumerate}

The vector space over $\mathbb{R}$ spanned by $A_i$'s forms a commutative algebra, denoted by $\mathcal{A}$ and called \emph{adjacency algebra}.
There exists a basis of $\mathcal{A}$ consisting of primitive idempotents, say $E_0=(1/n)J_n,E_1,\ldots,E_d$. 
Since  $\{A_0,A_1,\ldots,A_d\}$ and $\{E_0,E_1,\ldots,E_d\}$ are two bases of $\mathcal{A}$, there exist the change-of-bases matrices $P=(P_{ij})_{i,j=0}^d$, $Q=(Q_{ij})_{i,j=0}^d$ so that
\begin{align*}
A_j=\sum_{i=0}^d P_{ij}E_i,\quad E_j=\frac{1}{n}\sum_{i=0}^d Q_{ij}A_i.
\end{align*}
The matrix $P$ ($Q$ respectively) is said to be the {\em first (second respectively) eigenmatrix}.

\section{Unbiased orthogonal designs}\label{sec:uod}
\begin{definition}\label{def:uod}
Let $D_1,D_2$ be orthogonal designs of order $n$ and type $(s_1,\ldots,s_u)$ in variables $x_1,\ldots,x_u$. 
The orthogonal designs $D_1,D_2$ are {\em unbiased with parameter $\alpha$} if $\alpha$ is a positive real number and there exists a $(0,1,-1)$-matrix $W$ such that 
\begin{align*}
D_1D_2^\top=\frac{s_1x_1^2+\cdots+s_ux _u^2}{\sqrt{\alpha}}W.
\end{align*}  
Orthogonal designs $D_1,\ldots,D_f$ are {\em mutually unbiased with parameter $\alpha$} if any distinct two of the orthogonal designs are unbiased with parameter $\alpha$. 
\end{definition}
\begin{remark}
\begin{enumerate}
\item Note that the $(0,1,-1)$-matrix $W$ in Definition~\ref{def:uod} must be a weighing matrix of weight $\alpha$, thus $\alpha$ must  be a positive integer. 
\item  If $\alpha=n$, then the matrix $W$ in Definition~\ref{def:uod} is a Hadamard matrix of order $n$.
\end{enumerate}
\end{remark}

\begin{proposition}\label{prop:uod}
Suppose there exist unbiased orthogonal designs of order $n$ and type $(s_1,\ldots,s_u)$ with parameter $\alpha$. 
Then there exist quasi-unbiased weighing matrices for the parameters $(n,\sum_{i\in S}s_i,\alpha,(\sum_{i\in S}s_i)^2/\alpha))$
 for any nonempty subset $S\subset \{1,\ldots,u\}$.    
\end{proposition}
\begin{proof}
Let $D_1,D_2$ be unbiased orthogonal designs with the desired parameters. 
Substituting  $1$ if $j\in S$ and $0$ otherwise into $x_j$ in $D_1$ and $D_2$ yield quasi-unbiased weighing matrices for the desired parameters. 
\end{proof}

\begin{remark}
Proposition~\ref{prop:uod} shows that we have:
\begin{itemize}
\item unbiased Hadamard matrices if $\alpha=n$ and  $\sum_{i=1}^u s_i=n$, 
\item unbiased weighing matrices if $\alpha<n$ and $S\subseteq\{1,\ldots,u\}$ such that $\sum_{i\in S}s_i=\alpha$,
\item quasi-unbiased Hadamard matrices if $\alpha<n$ and $\sum_{i=1}^u s_i=n$,
\item quasi-unbiased weighing matrices if $S\subseteq \{1,\ldots,u\}$ such that $\sum_{i\in S}s_i\neq\alpha$.
\end{itemize}
Thus, unbiased orthogonal designs is a unified concept for various unbiased matrices.
\end{remark}
Assume that $D_1,\ldots,D_f$ are mutually unbiased orthogonal designs of order $n$ and type $(s_1,\ldots,s_u)$ with parameter $\alpha$. 
By Proposition~\ref{prop:uod} with $S=\{1\}$, we obtain quasi-unbiased weighing matrices $W_1,\ldots,W_f$ for the parameters $(n,s_1,\alpha,s_1^2/\alpha)$. 
Then $(\sqrt{\alpha}/s_1)W_1W_2^\top,\ldots,(\sqrt{\alpha}/s_1)W_1W_f^\top$ are $f-1$ mutually unbiased weighing matrices of weight $\alpha$.  
Applying \cite[Corollary~9]{BKR} to these, we obtain the following upper bound. 
\begin{proposition}\label{prop:UPuod}
Let $D_1,\ldots,D_f$ be mutually unbiased orthogonal designs of order $n$ and type $(s_1,\ldots,s_u)$ with parameter $\alpha$. 
Then the following holds.
\begin{enumerate}
\item $f\leq \frac{(n-1)(n+2)}{2}+1$. 
\item If $3\alpha-(n+2)\geq 0$, then $f\leq \frac{\alpha(n-1)}{3\alpha-(n+2)}+1$.  
\end{enumerate} 
\end{proposition}
\begin{problem}
Find examples of mutually unbiased orthogonal designs attaining the upper bounds in Proposition~\ref{prop:UPuod}, or improve the upper bounds. 
\end{problem}
In the rest of this section we show how constructions of unbiased Hadamard/weighing matrices are extended 
 to those of unbiased orthogonal designs. 

The direct product of matrices is used to give unbiased orthogonal designs. 
\begin{proposition}
If there exist unbiased orthogonal designs $D_1,\ldots,D_f$ of order $n$ with parameter $\alpha$ and type $(s_1,\ldots,s_u)$ and unbiased orthogonal designs $D'_1,\ldots,D'_f$ of order $m$ and type $(s_1,\ldots,s_u)$ with parameter $\alpha$, then $D_1\oplus D'_1,\ldots,D_f\oplus D'_f$ are unbiased orthogonal designs of order $n+m$ and type $(s_1,\ldots,s_u)$ with parameter $\alpha$. 
\end{proposition}
\begin{proof}
Straightforward. 
\end{proof}

The tensor product of unbiased orthogonal designs and quasi-unbiased weighing matrices give unbiased orthogonal designs. 
\begin{proposition}
Suppose that there exist  unbiased orthogonal designs $D_1,\ldots,D_f$ of order $n$ and type $(s_1,\ldots,s_u)$ with parameter $\alpha$ and quasi-unbiased weighing matrices $W_1,\ldots,W_f$ for the parameters $(m,k,l,a)$. 
\begin{enumerate}
\item  $D_1\otimes W_1,\ldots,D_f\otimes W_1$ are unbiased orthogonal designs of order $nm$ and type $(k s_1,\ldots,k s_u)$  with parameter $\alpha$.
\item $D_1\otimes W_1,\ldots,D_1\otimes W_f$ are unbiased orthogonal designs of order $nm$ and type $(k s_1,\ldots,k s_u)$  with parameter $l\alpha$.
\end{enumerate}
\end{proposition}  
\begin{proof}
Straightforward. 
\end{proof}

For an orthogonal design of order $n$ and type $(s_1,\ldots,s_u)$ and a weighing matrix  of order $n$ and weight $k$, we obtain $n$ matrices as an extension of  a part of a result in \cite{K}. 
\begin{lemma}\label{lem:ah}
Let $D$ be an orthogonal design $D$ of order $n$ and type $(s_1,\ldots,s_u)$ in variables $x_1,\ldots,x_u$ with $i$-th column $d_i$, and $W$ a weighing matrix of order $n$ and weight $k$ with $i$-th column $w_i$. 
Define $C_i=w_i d_i^\top,W_i=w_i w_i^\top$ for $i\in\{1,\ldots,n\}$. 
Then the following hold. 
\begin{enumerate}[(1)]
\item $C_i C_j^\top=O_n$, $1\leq i\neq j\leq n$, where $O_n$ is the zero matrix.
\item $C_i C_i^\top=(s_1 x_1^2+\cdots+s_u x_u^2)W_i$, $1\leq i\leq n$. 
\item $\sum_{i=1}^n W_i=k I_n$. 
\end{enumerate} 
\end{lemma}
The following lemma will be used in Proposition~\ref{prop:msls}, which is an extension of a construction of Bush-type Hadamard matrices \cite[Corollary 5]{K}, see Section~\ref{sec:ff} for the definition of Bush-type Hadamard matrices. 
\begin{lemma}\label{lem:ls}
Let $D,W,C_i$ be the same as Lemma~\ref{lem:ah},  and $L=(l(i,j))_{i,j=1}^n$ a Latin square of order $n$.
Then $\tilde{D}=(C_{l(i,j)})_{i,j=1}^n$ is an orthogonal design of order $n^2$ and type $(k s_1,\ldots,k s_u)$.
\end{lemma}
\begin{proof}
The $(i,j)$-block of $\tilde{D} \tilde{D}^\top$ is
\begin{align}\label{eq:ls}
\sum_{m=1}^n C_{l(i,m)}C^\top_{l(j,m)}. 
\end{align}
When $i=j$, \eqref{eq:ls} is equal to $k(s_1 x_1^2+\cdots+s_u x_u^2)I_n$ by Lemma~\ref{lem:ah} (2), (3). 
When $i\neq j$, \eqref{eq:ls} is equal to $O_n$ by Lemma~\ref{lem:ah} (1). 
Thus $\tilde{D}$ is an orthogonal design of order $n^2$ and type $(k s_1,\ldots,k s_u)$. 
\end{proof}

Next we use Latin squares.  
Two Latin squares $L_1$ and $L_2$ of size $n$ on symbol set $\{1,2,\ldots,n\}$ are called
{\it suitable} if every superimposition of each row of $L_1$ on each row of $L_2$ results in
only one element of the form $(a,a)$.
Latin squares in which every distinct pair of Latin squares
is suitable are called \emph{mutually suitable Latin squares}. 
Note that the existence of $f$ mutually suitable Latin squares is equivalent to the existence of $m$ mutually orthogonal Latin squares \cite[Lemma 9]{hko}. 
The following is an extension of \cite[Theorem~13]{hko}.
\begin{proposition}\label{prop:msls}
If there exist an orthogonal design $D$ of order $n$ and type $(s_1,\ldots,s_u)$, a weighing matrix $W$ of order $n$ and weight $k$ and $f$ mutually suitable Latin squares $L_1,\ldots,L_f$ of order $n$, 
then  there exist $f+1$ unbiased orthogonal designs of order $n^2$ and type $(k s_1,\ldots,k s_u)$ with parameter $\alpha=1$.
\end{proposition}
\begin{proof}
Let $m_1,m_2$ be distinct elements in $\{1,\ldots,f\}$. 
Let $l(i,j),l'(i,j)$ denote the $(i,j)$-entry of $L_{m_1},L_{m_2}$ respectively. 
Set $\tilde{D}_{m_1}=(C_{l(i,j)}), \tilde{D}_{m_2}=(C_{l'(i,j)})$, where $C_i$ is defined in Lemma~\ref{lem:ah}. 
By Lemma~\ref{lem:ls}, each $\tilde{D}_i$ is an orthogonal design of order $n^2$ and type $(k s_1,\ldots,k s_u)$. 

First we claim $\tilde{D}_{m_1},\tilde{D}_{m_2}$ are unbiased with parameter $\alpha=1$. 
We calculate the $(i,j)$-block of $\tilde{D}_{m_1} \tilde{D}_{m_2}^\top$ as follows.
\begin{align}\label{eq:msls}
\text{the $(i,j)$-block of }\tilde{D}_{m_1}\tilde{D}_{m_2}^\top=\sum_{m=1}^n C_{l(i,m)}C^\top_{l'(j,m)}.
\end{align}
There uniquely exists $k\in\{1,\ldots,n\}$ such that $l(i,k)=l'(j,k)=a$, say,  and $l(i,m)\neq l'(j,m)$ for any $m\neq k$ since $L_{m_1},L_{m_2}$ are suitable. 
Then \eqref{eq:msls} is 
\begin{align*}
C_{l(i,k)}C^\top_{l'(j,k)}=C_a C_a^\top=(s_1 x_1^2+\cdots+s_u x_u^2)W_a.
\end{align*}
Since $W_a$ is a $(0,1,-1)$-matrix, $\tilde{D}_{m_1},\tilde{D}_{m_2}$ are unbiased with parameter $\alpha=1$. 

Next we show that one more orthogonal design is added as follows. 
Define a $(0,\pm x_1,\ldots,\pm x_u)$-matrix $D'$ to be $(w_j d_i^\top)_{i,j=1}^n$. 
Then $D'$ is an orthogonal design. Indeed, 
\begin{align*}
\text{the $(i,j)$-block of }D'D'^\top=\sum_{m=1}^n w_m d_i^\top d_j w_m^\top=\delta_{ij}k(s_1x_1^2+\cdots+ s_u x_u^2)I_n.
\end{align*} 
Next we show that $D'$ and $\tilde{D}_m$ are unbiased for any $m\in\{1,\ldots,f\}$. 
Letting $l''(i,j)$ denote the $(i,j)$-entry of a Latin square $L_m$, there uniquely exists $k\in\{1,\ldots,n\}$ such that $l''(j,k)=i$, where .
Then
\begin{align*}
\text{the $(i,j)$-block of }D'\tilde{D}_m^\top&=\sum_{m=1}^n w_m d_i^\top d_{l''(j,m)} w_{l''(j,m)}^\top\\&= (s_1x_1^2+\cdots+ s_u x_u^2)w_k w_{i}^\top. 
\end{align*}
Since $w_k w_{i}^\top$ is a $(0,1,-1)$-matrix, $D'$ and $\tilde{D}_{m}$ are unbiased with parameter $\alpha=1$. 
\end{proof}

It is known that if there exist Hadamard matrices of order $4m,4n$, then there exists a Hadamard matrix of order $8mn$~\cite{A85}. 
This construction was used to construct quasi-unbiased Hadamard matrices in \cite{AHS}.
We use this idea to orthogonal designs in order to obtain unbiased orthogonal designs.
\begin{proposition}
If there exist an orthogonal design order $4m$ of type $(s_1,\ldots,s_u)$ and 
quasi-unbiased Hadamard matrices for parameters $(4n,4n,l,a)$, then there exist unbiased orthogonal designs of order $8mn$ and type  $(2n s_1,\ldots,2n s_u)$ with parameter $\alpha=16n^2/a$.  
\end{proposition}
\begin{proof}
Let $D$ be an orthogonal designs of order $4m$ and type $(s_1,\ldots,s_u)$ in variables $x_1,\ldots,x_u$ and $H_1,H_2$ be quasi-unbiased Hadamard matrices for the parameters $(4n,4n,l,a)$.

Let $H_{i,j}$ ($i,j=1,2$) be $4n\times 2n$ matrices and $D_i$ ($i=1,2$) be $2m\times 4m$ matrices such that
\begin{align*}
H_i=\begin{pmatrix}H_{i,1}&H_{i,2}\end{pmatrix}, \quad D=\begin{pmatrix}D_1 \\D_2\end{pmatrix}.
\end{align*}


We define $\tilde{D}_i$ ($i=1,2$) as
\begin{align*}
\tilde{D}_i=\frac{1}{2}(H_{i,1}+H_{i,2})\otimes D_1+\frac{1}{2}(H_{i,1}-H_{i,2})\otimes D_2.
\end{align*}  
Then it is directly shown that  $\tilde{D}_i$ ($i=1,2$) are unbiased orthogonal designs of order $8nm$ and type $(2n s_1,\ldots,2n s_u)$ with parameter $\alpha=16n^2/a$.
\end{proof}

\section{A construction and some applications using the plug-in method}\label{sec:app}
In this section, first we provide a construction of quasi-unbiased weighing matrices, and then it will be used to construct unbiased orthogonal designs. 

The following lemma is a construction of $(0,1)$-matrices from a finite ring with unity, which satisfy Lemma~\ref{lem:01m}. 
This lemma with Lemma~\ref{lem:01m} constructs mutually quasi-unbiased weighing matrices satisfying Proposition~\ref{prop:uod2}. 

\begin{lemma}\label{lem:01mr}
Let $R$ be a finite ring with unity and $n$ elements. 
If there exist elements $x_1,\ldots,x_m\in R$ such that $x_i-x_j$ is a unit in $R$ for any distinct $i,j$, then there exist $n\times n$ monomial $(0,1)$-matrices $K_{i,j}$, $i,j\in\{1,\ldots,m\}$ such that $\sum_{l=1}^m K_{i,l}K_{j,l}^\top$ is a $(0,1)$-matrix for any distinct $i,j$.  
\end{lemma}
\begin{proof}
Assume that the additive group of $R$ is isomorphic to $\mathbb{Z}_{n_1}\times\cdots \times \mathbb{Z}_{n_s}$. 
Let $r_h$
be an $h\times h$ circulant matrix with the first row $(0,1,0,\ldots,0)$. 
We identify elements in $R$ with elements in $\mathbb{Z}_{n_1}\times\cdots \times \mathbb{Z}_{n_s}$.
Define a group homomorphism $\phi:R\rightarrow GL_{n}(\mathbb{R})$ as $\phi((x_i)_{i=1}^s)=\otimes_{i=1}^t r_{n_i}^{x_i}$. 

Let $\alpha_1,\ldots,\alpha_m$ be any distinct elements in $R$.
Set $K_{i,j}=\phi(x_i \alpha_j)$ for $i,j\in\{1,\ldots,m\}$. 
Then each $K_{i,j}$ is clearly an $n\times n$ monomial $(0,1)$-matrix. 
For any distinct $i,j$, 
$\sum_{l=1}^{m} K_{i,l}K_{j,l}^\top=\sum_{l=1}^{m}\phi((x_i-x_j)\alpha_l)$ is a $(0,1)$-matrix since $x_i-x_j$ is a unit for any distinct $i,j$. 
\end{proof}

\begin{example}
Let $m,p$ be positive integers such that $p$ is the least prime number dividing $m$. 
Then $0,1,\ldots,p-1$ satisfy the property that the difference of any distinct two elements is a unit in $\mathbb{Z}_m$, and $p$ is the maximum number of such elements by the pigeonhole principle. 
\end{example}

\begin{example}\label{exa:ff}
Let $q$ be a prime power and $\mathbb{F}_{q}$ the finite field with $q$ elements. 
Then any distinct $m$ elements in $\mathbb{F}_{q}$ satisfy that the difference of any distinct two is a unit in $\mathbb{F}_{q}$.  
\end{example}
 
\begin{example}
Let $p,s,m$ be positive integers such that $p$ is prime,  let $h(x)$ be a basic irreducible polynomial of degree $m$ over $\mathbb{Z}_{p^s}$. The ring $\mathbb{Z}_{p^s}[x]/(h(x))$ is called a {\em Galois ring}, denoted by $GR(p^s,p^{sm})$. 
Write $\xi=x+(h(x))$. Then the order of $\xi$ is $p^{m}-1$, and $\xi^i-\xi^j$ is a unit for any distinct $i,j\in\{0,\ldots,p^m-2\}$ \cite[Theorem 14.8]{W}.
\end{example}

We pose a problem in order to construct $(0,1)$-matrices in Lemma~\ref{lem:01mr}. 
\begin{problem}
For a given finite ring  $R$ with unity, determine the largest positive integer $m$ such that there exist elements $x_1,\ldots,x_m\in R$ in such a way that $x_i-x_j$ is a unit in $R$ for any distinct $i,j$.  
\end{problem}

For an $m\times m$ matrix $W=(w_{ij})_{i,j=1}^m$ and $n\times n$ matrices $K_1,\ldots,K_m$, 
denote by $W\otimes (K_1,\ldots,K_m)$  
\begin{align*}
\begin{pmatrix}
w_{11}K_1 & w_{12}K_2 & \cdots & w_{1m}K_m\\
w_{21}K_1 & w_{22}K_2 & \cdots & w_{2m}K_m\\
\vdots & \vdots & \ddots & \vdots\\
w_{m1}K_1 & w_{m2}K_2 & \cdots & w_{mm}K_m
\end{pmatrix}.
\end{align*}

The following lemma provides mutually quasi-unbiased weighing matrices from $(0,1)$-matrices satisfying the assumptions of Lemma~\ref{lem:01mr}, and will be used to construct unbiased orthogonal designs in Proposition~\ref{prop:uod2}.
\begin{lemma}\label{lem:01m}
Let $n,m,k$ be positive integers such that $m\leq n$. 
Let $W=(w_{ij})_{i,j=1}^{m}$ be a weighing matrix of order $m$ and weight $k$,  $K_{i,j}$ ($i,j\in\{1,\ldots,m\}$) $n\times n$ monomial $(0,1)$-matrices such that $\sum_{l=1}^m K_{i,l}K_{j,l}^\top$ is a $(0,1)$-matrix for any distinct $i,j\in\{1,\ldots,m\}$, and set $W_i=W\otimes (K_{i,1},\ldots,K_{i,m})$ for $i\in\{1,\ldots,m\}$.   
Then the following hold.
\begin{enumerate}
\item $W_i$ is a weighing matrix of order $n m$ and weight $k$.
\item $W_1,\ldots, W_m$ are mutually quasi-unbiased weighing matrices for $(nm,k,k^2,1)$. 
\end{enumerate}
\end{lemma}
\begin{proof}
(1): 
Since the $(a,b)$-th block of $W_iW_i^\top$ is 
\begin{align*}
\sum_{l=1}^m (w_{al}K_{i,l})(w_{bl}K_{i,l})^\top=\sum_{l=1}^m w_{al}w_{bl}I_n=\delta_{ab}k I_n, 
\end{align*}
$W_i$ is a weighing matrix of the desired order and weight, where $\delta_{ab}$ denotes the Kronecker delta.

(2): 
It is enough to show that $W_iW_j^\top$ is a $(0,1,-1)$-matrix for any distinct $i,j$. 
Letting $i,j$ be distinct elements in $\{1,\ldots,m\}$, the $(a,b)$-th block of $W_iW_j^\top$ is 
\begin{align*}
\sum_{l=1}^m(w_{al}K_{i,l})(w_{bl}K_{j,l})^\top =\sum_{l=1}^m w_{al}w_{bl}K_{i,l}K_{j,l}^\top,
\end{align*}
which is a $(0,1,-1)$-matrix, since $\sum_{l=1}^m K_{i,l}K_{j,l}^\top$ is a $(0,1)$-matrix and $w_{al}w_{bl}$ is $0,\pm1$. 
\end{proof}

Finally we provide a construction for unbiased orthogonal designs from some quasi-unbiased weighing matrices and an orthogonal design. 
\begin{proposition}\label{prop:uod2}
Let $W_1,\ldots,W_f$ be mutually quasi-unbiased weighing matrices for parameters $(nm,k,k^2,1)$.
Assume that $W_i(I_n\otimes J_m)$ is a $(0,1,-1)$-matrix for any $i\in\{1,\ldots,f\}$. 
Let $K$ be an orthogonal design of order $m$ and type $(s_1,\ldots,s_u)$
 in variables $x_1,\ldots,x_u$. 
Then there exist $f$ mutually unbiased orthogonal designs of order $nm$ and type $(k s_1,\ldots,k s_u)$ with parameter $\alpha=k^2$. 
\end{proposition}
\begin{proof}
Let $D_i=W_i(I_n\otimes K)$ for any $i\in\{1,\ldots,f\}$.

Each matrix $D_i$ is clearly a $(0,\pm x_1,\ldots,\pm x_u)$-matrix. 
For $i,j\in\{1,\ldots,f\}$, 
\begin{align*}
D_i{D}_j^\top&=W_i W_j^\top (I_n\otimes KK^\top)
=(\sum_{k=1}^u s_k x_k^2)W_iW_j^\top.
\end{align*}
Since $W_iW_i^\top=k I_{nm}$ for any $i$, $D_i$ is an orthogonal design of order $nm$ and type $(k s_1,\ldots,k s_u)$.   
Since $W_iW_j^\top$ is a $(0,1,-1)$-matrix for any distinct $i,j$,  $D_i,D_j$ are unbiased with parameter $\alpha=k^2$. 
\end{proof}

We are ready for the main result. 
By Lemmas~\ref{lem:01mr}, \ref{lem:01m}, Propositions~\ref{prop:uod}, \ref{prop:uod2} and Example~\ref{exa:ff}, we obtain the following result.
\begin{theorem}\label{thm:constuod}
Let $q,m,k,s_1,\ldots,s_u$ be positive integers such that $q$ is a prime power and $m\leq q$.
Assume that there exist a weighing matrix of order $m$ and weight $k$ and an orthogonal design of order $m$ and type $(s_1,\ldots,s_u)$. Then the following hold.
\begin{enumerate}  
\item There exist $m$ mutually unbiased orthogonal designs of order $m q$ and type $(k s_1,\ldots,k s_u)$  with parameter $\alpha=k^2$. 
\item There exist $m$ mutually quasi-unbiased weighing matrices for parameters $(m q,k\sum_{i\in S}s_i,k^2,(\sum_{i\in S}s_i)^2)$ for any nonempty subset $S\subset \{1,\ldots,u\}$. 
\end{enumerate}
\end{theorem}
  
In particular, by taking $q$ a power of $2$ in Theorem~\ref{thm:constuod} and by Lemma~\ref{lem:odw}, 
we obtain the following.
\begin{corollary}\label{cor:const}
Let $t,k,u,s_1,\ldots,s_{u},m$ be positive integers such that $k,m\leq 2^t$, $u$ is $2^t$ if $t=1,2,3$ and $2t$ if $t>3$  and 
\begin{align*}
(s_i)_{i=1}^{2}&=(1,1) \text{ if } t=1,\\
(s_i)_{i=1}^{4}&=(1,1,1,1) \text{ if } t=2,\\
(s_i)_{i=1}^{8}&=(1,1,1,1,1,1,1,1) \text{ if } t=3,\\
(s_i)_{i=1}^{2t}&=(1,1,1,1,2,2,4,4,\ldots,2^{t-2},2^{t-2}) \text{ if } t>3.  
\end{align*}
Then the following hold.
\begin{enumerate}
\item There exist  $2^t$ mutually unbiased orthogonal designs of order $2^{2t}$ and type $(k s_i)_{i=1}^{2t}$ with parameter $\alpha=k^2$. 
\item There exist $2^t$ mutually quasi-unbiased  weighing matrices for the parameters $(2^{2t},m k, k^2,m^2)$.  
\end{enumerate}
\end{corollary}

In the rest of this section we use the {\it plug-in method} in Theorem~\ref{thm:constuod} in order to show some of the many applications of the construction there. In order to use the plug-in method in Theorem~\ref{thm:constuod} the variables should be replaced with {\it amicable matrices} and in order to preserve the orthogonality of the designs, the matrices should satisfy the {\it sum property}. For example, matrices $A$ and $B$ replaces variables $a$ and $b$, if $A$ and $B$ are amicable, i.e. $AB^t=BA^t$. The sum property refers to the property that matrices  $A_i$ replacing variables $a_i$, $i=1,2,\cdots,k$, should satisfy $\sum_{i=1}^{i=k}A_iA_i^t=\ell I$ for some positive integer $\ell$.
We refer reader to \cite{seb-yam} for the terminologies not defined here. Our first application relates to part (1) in Corollary \ref{cor:const}, but we need to recall a result of Goethals and Seidel \cite{gs-72}. There they showed the existence of two circulant and symmetric $(1,-1)$-matrices $I_q+R$ and $S$ of order $q=\frac{1}{2}(p+1)$, $p\equiv 1\pmod{4}$ a prime power such that $RR^\top+SS^\top=pI_q$. Note that the existence of Goethals-Seidel matrices imply the existence of Williamson matrices, see \cite{gs-72}.  
\begin{proposition}
There are two quasi-unbiased  weighing matrices for the parameters $(4q,4q-2,4,(2q-1)^2)$ for every $q=\frac{1}{2}(p+1)$, $p\equiv 1\pmod{4}$ a prime power. 
\end{proposition}
\begin{proof}
 Replace the variables in part (1) of Corollary \ref{cor:const} by Goethals-Seidel matrices of order  $q=\frac{1}{2}(p+1)$, $p\equiv 1\pmod{4}$ a prime power.
\end{proof}
Our second application relates to part (2) in Corollary \ref{cor:const}, where there are four independent variables. Here we replace the variables by four Williamson type matrices. 
\begin{proposition}\label{prop:16n}
 There are four mutually quasi-unbiased  Hadamard matrices for the parameters $(16n,16n,16,16n^2)$ for every $n$ which is the order of 
Williamson type matrices. 
\end{proposition}
\begin{proof}
 Replace the variables in part (2) of Corollary \ref{cor:const} by the Williamson type matrices of order  $n$.
\end{proof}
\begin{corollary}
There are four quasi-unbiased  weighing matrices for the parameters $(16q,16q,16,16q^2)$ for every $q=\frac{1}{2}(p+1)$, $p\equiv 1\pmod{4}$ a prime power.
\end{corollary}
\begin{proof}
There are Williamson matrices of order $q=\frac{1}{2}(p+1)$, $p\equiv 1\pmod{4}$ a prime power, so this follows from Proposition \ref{prop:16n}.
\end{proof}
\begin{example}
Since there exist Williamson type matrices of order $3$, Proposition~\ref{prop:16n} provides four mutually quasi-unbiased Hadamard matrices for the parameters $(48,48,16,144)$ are obtained. 
These parametrs were missing from  Table~1 in  \cite{AHS}. 
\end{example}

Our last application relates to the asymptotic existence of quasi-unbiased  Hadamard matrices. In order to do this we need the following important and well known result, see \cite{gs-72, js-asy}.
\begin{lemma}\label{s-75}
There is an $OD(2^t;a,b,2^t-a-b)$ for all integers $t\ge 2$ and $0\le a+b \le 2^t$. 
\end{lemma}
\begin{proof}
 This is Corollary 7.2 of \cite{gs-72}.
\end{proof}
Seberry used the most effective and intelligent use of the above lemma in order to show the asymptotic existence of Hadamard matrices. The essence of her method is that for a give prime number $q$ it is sufficient to work with only three $(\pm 1)$-matrices, namely, $J_q$, $J_q-2I_q$ and $P_q$, the Paley matrix of order $q$. In order to make these matrices mutually amicable, all that is needed is to multiply $P_q$ on the right by $R_q$, the back identity matrix of order $q$. For a given prime number $q$ she found a positive integer $t$ for which $2^t$ can be written as sum of three {\it suitable} positive integers, which was determined in a way to have a  class of matrices chosen from $\{J_q, J_q-2I_q,P_q+I_q\}$ suitable for the plug-in.  The most important aspect of the method is the existence of an orthogonal design of order $2^t$ of a type determined in a way to make the selected matrices satisfy the sum property. She then used the corresponding orthogonal design of order $2^t$ in three variables and replaced the variables with the three plug-in matrices leading to the construction of a Hadamard matrix of order $2^tp$. The final step in her construction was to split any given integer as a product of prime numbers and using some product properties,  see \cite{js-asy} for details. Inspired by this method, we are led to the following  general result.  
\begin{theorem}\label{asy}
 Given a prime number $q$, there is some integer $t$ for which there are $2^t$ mutually quasi-unbiased  Hadamard matrices for the parameters $(2^{2t}q,2^{2t}q,2^{2t},2^{2t}q^2)$.
\end{theorem}
\begin{proof}
Following Seberry's method in \cite{js-asy}, we give a proof  for each of  $q\equiv 3\pmod{4}$ and $q\equiv 1\pmod{4}$. 
Let $q\equiv 3\pmod{4}$, there is $t$ depending on $q$ for which $2^t$ can be written as sum of three integers $a,b,c$ appropriate for suitable plug-in matrices. Applying  part (1) of  Theorem \ref{thm:constuod} we have $2^t$ mutually unbiased orthogonal designs of order $2^{2t}$ and type $(a,b,c)$. We then replace the variables with the appropriate plug-in matrices from $\{J_q, J_q-2I_q,P_q+I_q\}$. Note that the Paley matrix for $q\equiv 3\pmod{4}$ is skew symmetric. For $q\equiv 1\pmod{4}$, the Paley matrix $P_q$ is symmetric and thus both $P_q+I_q$ and $P_q-I_q$ should be present in the construction.  So, there is a need to add one more variable to the orthogonal design.  Applying  part (1) of  Theorem \ref{thm:constuod} we have $2^t$ mutually unbiased orthogonal designs of order $2^{2t}$ and type $(a,b,c,c)$. We now replace the variables with the appropriate plug-in matrices from $\{J_q, J_q-2I_q,P_q+I_q,P_q-I_q\}$.
\end{proof}
In the following two examples we illustrate the construction method in Theorem \ref{asy}.
\begin{example}
Let $q=5$ in Theorem \ref{asy}. Starting with the $OD(8;1,1,6)$, let $D$ be the $OD(16;2,2,6,6)$. By Theorem \ref{thm:constuod}, there are $16$ mutually unbiased orthogonal designs of order $256$ and type $(32,32,96,96)$ in variables $a$, $b$, $c$ and $d$, respectively. Replacing $a$ with $J_5$, $b$ with $J_5-2I_5$, $c$ with $P_5+I_5$ and $d$ with $P_5-I_5$,  we get $16$ mutually quasi-unbiased Hadamard matrices of order $1280$ for the parameters $( 1280,1280,256,6400)$.
\end{example}
\begin{example}
Let $q=7$ in Theorem \ref{asy}. Let $D$ be the $OD(16;1,3,13)$. By Theorem \ref{thm:constuod}, there are $16$ mutually unbiased orthogonal designs of order $256$ and type $(16,32,208)$ in variables $a$, $b$ and $c$ respectively. Replacing $a$ with $J_7$, $b$ with $J_7-2I_7$ and $c$ with the converted Paley matrix $P_7R_7$ of order $7$ (the Paley matrix $P_7$ is multiplied by the back identity matrix $R_7$ of order $7$), we get $16$ mutually quasi-unbiased Hadamard matrices of order $1792$ for the parameters $( 1792,1792,256,12544)$.
\end{example}

\section{Quasi-unbiased weighing matrices for $(2^{2t},2^t,2^{2t},1)$}\label{sec:ff}
In this section, we focus on quasi-unbiased weighing matrices for the parameters $(2^{2t},2^t,2^{2t},1)$
 in Corollary~\ref{cor:const} (2) for the case where $W=(w_{ij})_{i,j=1}^{2^t}$ is a Hadamard matrix of order $2^t$. 
Recall that $\phi$ is the group homomorphism from the additive group $\mathbb{F}_{2^t}$ to $GL_{2^t}(\mathbb{R})$, and $W_i=(w_{kl}\phi(\alpha_i\alpha_l))_{k,l=1}^{2^t}$ where $\mathbb{F}_{2^t}=\{\alpha_1,\ldots,\alpha_{2^t}\}$. 
Let $W_{2^t+1}=I_{2^t}\otimes W$.
We are going to show that 
\begin{itemize}
\item $W_1,\ldots,W_{2^t}$ yield maximal mutually unbiased Bush-type Hadamard matrices, 
\item $W_1,\ldots,W_{2^t+1}$ are maximal,
\item $W_1,\ldots,W_{2^t+1}$ give rise to an association scheme.
\end{itemize}

A Hadamard matrix $H$ of order $n^2$ is of {\em Bush-type} if $H$ is partitioned into $H=(H_{ij})_{i,j=1}^n$,  $n^2$ squares of size $n$, such that $H_{ii}=J_n$ for any $i\in\{1,\ldots,n\}$ and $H_{ij}J_n=J_nH_{ij}=O_n$ for any distinct  $i,j\in\{1,\ldots,n\}$ \cite{Bush}.
\begin{proposition}\label{prop:bush}
The matrix $W_i W_j^\top$ is a Bush-type Hadamard matrix of order $2^{2t}$ for any distinct $i,j\in\{1,\ldots,2^t\}$.
\end{proposition}
\begin{proof}
Let $i,j$ be any distinct elements in $\{1,\ldots,2^t\}$.  
The product $W_i W_j^\top$ is calculated as follows:
\begin{align}\label{eq:uod}
\text{the $(a,b)$-block of }W_i W_j^\top&=\sum_{m=1}^{2^t} w_{am}w_{bm}\phi((\alpha_i+\alpha_j)\alpha_m).
\end{align}
For any $a=b$, \eqref{eq:uod} is 
\begin{align*}
\sum_{m=1}^{2^t} \phi((\alpha_i+\alpha_j)\alpha_m)=\sum_{m=1}^{2^t}\phi(\alpha_m)=J_{2^t}. 
\end{align*}
For any $a\neq b$, $w_{am}w_{bm}$ takes $1,-1$ exactly $2^{t-1}$ times when $m$ runs over $\{1,\ldots,2^t\}$. 
Thus any row sum and any column sum of \eqref{eq:uod} are equal to $0$. 
Therefore $W_iW_j^\top$ is of Bush-type. 
\end{proof}
In particular $W_1W_2^\top,\ldots,W_1W_{2^t}^\top$ are mutually unbiased Bush-type Hadamard matrices of order $2^{2t}$ \cite{KSS}. 
The $2^t-1$ mutually unbiased Bush-type Hadamard matrices attain the upper bound in \cite[Remark 8 (b)]{KSS}

A set of mutually quasi-unbiased matrices $W_1,\ldots,W_f$ for parameters $(n,k,l,a)$ is {\em maximal} if there is no weighing matrix $W$ such that $W_1,\ldots,W_f,W$ are mutually quasi-unbiased for the same parameters. 
The {\em weight} of a vector $u$ is the number of the non-zero entries of $u$. 
A column $(0,1,-1)$-vector $u$ of weight $k$ and a weighing matrix $W$ of order $n$ and weight $k$ are {\em quasi-unbiased for parameters $(n,k,l,a)$} if $Wu$ is a $(0,\sqrt{a},-\sqrt{a})$-vector of weight $l$.  
Note that for quasi-unbiased weighing matrices $W_1,W_2$ for parameters $(n,k,l,a)$, $W_1$ and any column vector of $W_2^\top$ are quasi-unbiased for the parameters. 
\begin{theorem}\label{thm:maximal}
The set of mutually quasi-unbiased weighing matrices $W_1,\ldots,W_{2^t+1}$ for the parameters $(2^{2t},2^t,2^{2t},1)$ is maximal. 
\end{theorem}
\begin{proof}
We claim that there is no $(0,1,-1)$-vector of length $2^{2t}$ and weight $2^t$ which is quasi-unbiased for parameters $(2^{2t},2^t,2^{2t},1)$ to all of $W_1,\ldots,W_{2^t+1}$. 

Assume that there exists such a column vector $u=u_1\cdots u_{2^t}$ of length $2^{2t}$, where each $u_i$ has length $2^t$. 
Since $u$ is quasi-unbiased to $W_{2^t+1}$, each $u_i$ has exactly one non-zero coordinate which is in $\{1,-1\}$. 
Since 
\begin{align*}
W_1+\cdots+W_{2^t+1}\equiv 
\begin{pmatrix}
J_{2^t} & J_{2^t} & J_{2^t} & \cdots & J_{2^t}\\
O_{2^t} & O_{2^t} & J_{2^t} & \cdots & J_{2^t}\\
O_{2^t} & J_{2^t} & O_{2^t} & \cdots & J_{2^t}\\
\vdots & \vdots & \vdots & \ddots & \vdots\\
O_{2^t} & J_{2^t} & J_{2^t} & \cdots & O_{2^t}
\end{pmatrix} 
\pmod{2}
\end{align*}
and $J_{2^t}u_i\equiv {\bf 1} \pmod{2}$ where ${\bf 1}$ denotes the all-one vector, we obtain 
\begin{align}\label{eq:maximal}
(W_1+\cdots+W_{2^t+1})u=
\begin{pmatrix}
\sum_{j=1}^{2^t}J_{2^t}u_j\\
\sum_{j=2}^{2^t}J_{2^t}u_j-J_{2^t}u_2\\
\vdots\\
\sum_{j=2}^{2^t}J_{2^t}u_j-J_{2^t}u_{2^t}
\end{pmatrix}\equiv {\bf 0} \pmod{2}
\end{align}
where ${\bf 0}$ denotes the zero vector. 

On the other hand, since $u$ is quasi-unbiased to $W_i$ for any $i\in\{1,\ldots,2^t\}$, $W_i u$ is a $(1,-1)$-vector. 
Thus 
the entries of $W_1u+\cdots+W_{2^t+1}u$ are odd integers, a contradiction to \eqref{eq:maximal}. 
\end{proof}

Finally we show that an association scheme is obtained from any $f$ matrices  in $\{W_1,\ldots,W_{2^t+1}\}$. 
Let $W_{i_1},W_{i_2},\ldots,W_{i_f}$ be any $f$ matrices in $\{W_1,\ldots,W_{2^t+1}\}$. 
Let
 \begin{align*}G =
\begin{pmatrix}W_{i_1}\\ W_{i_2}\\ \vdots\\ W_{i_f}
\end{pmatrix}
\begin{pmatrix}W_{i_1}^\top&W_{i_2}^\top&\ldots&W_{i_f}^\top
\end{pmatrix}.
\end{align*}
Using the properties of the given matrices, we can write $G=2^t I_{f2^{2t}}+N$, where $N$
is an $f2^{2t}\times f2^{2t}$ $(1,-1)$-matrix. 

Write $N=N^+-N^-$, where
$N^+$ and $N^-$ are disjoint
$(0,1)$-matrices.
Then it holds that $N^++N^-=J_{2^{2t}}\otimes (J_f-I_f)$. 

Consider \begin{align*}
\tilde{G}=\begin{pmatrix}
G & -G \\
-G&G
\end{pmatrix}.
\end{align*}
We can write
\begin{align*}
\tilde{G}=
2^t I_{f2^{2t+1}}-2^tI_{f2^{2t}}\otimes(J_2-I_2)+
\begin{pmatrix}
N^+ & N^- \\
N^-&N^+
\end{pmatrix}-
\begin{pmatrix}
N^- & N^+ \\
N^+&N^-
\end{pmatrix}.
\end{align*}
Further we define $B$ to be a $(0,1)$-matrix obtained from $\tilde{G}$ by replacing zero entries with $1$ and non-zero entries with $0$, namely 
\begin{align*}
B=
\begin{pmatrix}
(J_{2^{2t}}-I_{2^{2t}})\otimes I_f & (J_{2^{2t}}-I_{2^{2t}})\otimes I_f\\
(J_{2^{2t}}-I_{2^{2t}})\otimes I_f & (J_{2^{2t}}-I_{2^{2t}})\otimes I_f
\end{pmatrix}. 
\end{align*}

Using the decomposition $J_{2^{2t}}-I_{2^{2t}}=(J_{2^t}-I_{2^t})\otimes I_{2^t}+J_{2^t}\otimes(J_{2^t}-I_{2^t})$, we let 
\begin{align*}
A_0 & =  I_{f2^{2t+1}},\qquad
A_1  =  (J_2-I_2)\otimes I_{f2^{2t}},\\
A_2 & =  \begin{pmatrix}N^+&N^-\\ N^- & N^+\end{pmatrix},\qquad
A_3  =  \begin{pmatrix}N^-&N^+\\ N^+ & N^-\end{pmatrix},\\
A_4 & = \begin{pmatrix}
(J_{2^t}-I_{2^t})\otimes I_{2^t} \otimes I_f & (J_{2^t}-I_{2^t})\otimes I_{2^t}\otimes I_f\\
(J_{2^t}-I_{2^t})\otimes I_{2^t}\otimes I_f & (J_{2^t}-I_{2^t})\otimes I_{2^t}\otimes I_f
\end{pmatrix},\\
A_5&=\begin{pmatrix}
J_{2^t}\otimes(J_{2^t}-I_{2^t})\otimes I_f & J_{2^t}\otimes(J_{2^t}-I_{2^t})\otimes I_f\\
J_{2^t}\otimes(J_{2^t}-I_{2^t})\otimes I_f & J_{2^t}\otimes(J_{2^t}-I_{2^t})\otimes I_f
\end{pmatrix}.
\end{align*}
The following lemma will be used in Theorem~\ref{thm:as}.
\begin{lemma}\label{lem:as1}
For any distinct $i,j\in\{1,\ldots,2^t+1\}$, let $W_iW_j^\top=(M_{ij})_{i,j=1}^{2^t}$ where each $M_{ij}$ is a $2^t\times 2^t$ matrix. 
Then $\sum_{m=1}^{2^t}M_{km}=\sum_{m=1}^{2^t}M_{mk}=2^t I_{2^t}$ for any $k\in\{1,\ldots,2^t+1\}$.
\end{lemma}
\begin{proof}
Straightforward. 
\end{proof}
\begin{theorem}\label{thm:as}
The set of matrices $\{A_0,A_1,\ldots,A_5\}$ forms an association scheme.  
\end{theorem}
\begin{proof}
Let $\mathcal{A}=\text{span}(A_0,A_1,\ldots,A_5)$. 
By \cite[Theorem 4.1]{LMO}, the set of matrices $\{A_0,A_1,A_2, A_3,A_4+A_5\}$ forms an association scheme. 

It is obvious that $A_1A_5,(A_2+A_3)A_5,A_4A_5,A_5^2\in\mathcal{A}$. 
By Lemma~\ref{lem:as1},  it holds that $(J_{2^t}\otimes I_{2^t}\otimes I_f)N=2^t(J_{2^t}-I_{2^t})\otimes I_{2^t}\otimes I_f$, from which we obtain $(A_0+A_1+A_5)(A_2-A_3)\in \mathcal{A}$. 
Thus it holds that $A_2A_5,A_3A_5\in\mathcal{A}$, therefore we conclude that $\mathcal{A}$ forms the adjacency algebra of an association scheme.  
\end{proof}
The eigenmatrices $P,Q$ of the association scheme with some ordering of $E_i$'s are given as follows:
\begin{align*}
P=&\begin{pmatrix}
1 & 1 & (f-1)2^{2t} & (f-1)2^{2t} & 2^{t+1}(2^t-1) & 2(2^t-1) \\
1 & 1 & -2^{2t} & -2^{2t} & 2^{t+1}(2^t-1) & 2(2^t-1) \\
1 & -1 & -2^t & 2^t & 0 & 0 \\
1 & -1 & (f-1)2^t & -(f-1)2^t & 0 & 0 \\
1 & 1 & 0 & 0 & -2^{t+1} & 2(2^t-1) \\
1 & 1 & 0 & 0 & 0 & -2 
\end{pmatrix},\\
Q=&\begin{pmatrix}
1 & f-1 & (f-1)2^{2t} & 2^{2t} & f(2^t-1) & f2^t(2^t-1) \\
1 & f-1 & -(f-1)2^{2t} & -2^{2t} & f(2^t-1) & f2^t(2^t-1) \\
1 & -1 & -2^t & 2^t & 0 & 0 \\
1 & -1 & 2^t & -2^t & 0 & 0 \\
1 & f-1 & 0 & 0 & -f & 0 \\
1 & f-1 & 0 & 0 & f(2^t-1) & -f2^t 
\end{pmatrix}.
\end{align*}
Note that the association scheme is uniform in the sense of \cite{dmm}.
\section{Unbiased unit orthogonal designs}\label{sec:uuod}
Finally we consider a generalization for unbiased orthogonal designs, namely we allow complex numbers for the entries. 

Let $\mathbb{T}=\{c\in\mathbb{C}\mid |c|=1\}$. 
A {\em unit weighing matrix of order $n$ and weight $k$} is an $n\times n$ $\{0\}\cup\mathbb{T}$-matrix $W$ such that $W W^*=kI_n$, where $W^*$ denotes the transpose conjugate of $W$. 

Two unit weighing matrices $W_1,W_2$ of order $n$ and weight $k$ are said to be {\em quasi-unbiased for  parameters $(n,k,l,a)$} if 
$(1/\sqrt{a})W_1 W_2^\top$ is a unit weighing matrix of order $n$ and weight $l$. 
Unit weighing matrices $W_1,\ldots,W_f$ of order $n$ and weight $k$ are {mutually quasi-unbiased for parameters $(n,k,l,a)$} if any distinct two of them are quasi-unbiased with the parameters.

A {\em unit orthogonal design of order $n$ and type $(s_1,\ldots,s_u)$ in variables $x_1,\ldots,x_u$} is an $n\times n$ matrix with entries in $\{0\}\cup\{\epsilon_i x_i\mid i=1,\ldots,u,\epsilon_i\in\mathbb{T} \}$, where $x_1,\ldots,x_u$ are distinct commuting real indeterminates, such that $D D^*=(s_1x_1^2+\cdots+s_u x_u^2)I_n$. 

Let $D_1,D_2$ be unit orthogonal designs of order $n$ and type $(s_1,\ldots,s_u)$ in variables $x_1,\ldots,x_u$. 
The unit orthogonal designs $D_1,D_2$ are {\em unbiased with parameter $\alpha$} if $\alpha$ is a positive real number and there exists an $n\times n$ $\{0\}\cup\mathbb{T}$-matrix $W$ such that 
\begin{align*}
D_1D_2^*=\frac{s_1x_1^2+\cdots+s_ux _u^2}{\sqrt{\alpha}}W.
\end{align*}  

We note that if we replace a weighing matrix with a unit weighing matrix or an orthogonal design with a unit orthogonal design in Section~\ref{sec:uod}, then we obtain unbiased unit orthogonal designs.  

We extend the construction in Example~\ref{exa:ff} to complex case. 
Considering a homomorphism from $\mathbb{Z}_m$ to $GL_m(\mathbb{C})$ defined by $x\mapsto c_m^x $  where $c_m=\text{diag}(1,w,w^2,\ldots,w^{m-1})\cdot r_m$ and $w$ is a primitive $m$-th root unity, we obtain the following lemma. 
The proof is obtained by replacing the homomorphism $\phi$ with the above.  
\begin{lemma}\label{lem:01um}
Let $q$ be a prime power. 
Then there exist $q\times q$ monomial $\{0\}\cup\mathbb{T}$-matrices $K_{i,j}$, $i,j\in\{1,\ldots,q\}$ such that $\sum_{l=1}^m K_{i,l}K_{j,l}^*=J_{q}$ for any distinct $i,j$.
\end{lemma}
A {\em complex Hadamard matrix of order $n$} is a unit weighing matrix of order $n$ and weight $n$. 
An {\em $(n,m)$-Butson Hadamard matrix} is a complex Hadamard matrix of order $n$ with entries equal to $m$-th roots of unity \cite{But}. 
A complex Hadamard matrix $H$ of order $n^2$ is of {\em Bush-type} if $H$ is partitioned into $H=(H_{ij})_{i,j=1}^n$,  $n^2$ squares of size $n$, such that $H_{ii}=J_n$ for any $i\in\{1,\ldots,n\}$ and $H_{ij}J_n=J_nH_{ij}=O_n$ for any distinct  $i,j\in\{1,\ldots,n\}$.
   
Using the matrices $K_{i,j}$'s in Lemma~\ref{lem:01um} with the same construction in Theorem~\ref{thm:constuod}, we obtain mutually unbiased unit orthogonal designs. 
\begin{theorem}\label{thm:constuuod}
Let $q,m,k,s_1,\ldots,s_u$ be positive integers such that $q$ is a prime power and $m\leq q$.
Assume that there exist a weighing matrix of order $m$ and weight $k$ and an orthogonal design of order $m$ and type $(s_1,\ldots,s_u)$. Then the following hold.
\begin{enumerate}  
\item There exist $m$ mutually unbiased unit orthogonal designs of order $m q$ and type $(k s_1,\ldots,k s_u)$  with parameter $\alpha=k^2$. 
\item There exist $m$ mutually quasi-unbiased unit weighing matrices for parameters $(m q,k\sum_{i\in S}s_i,k^2,(\sum_{i\in S}s_i)^2)$ for any nonempty subset $S\subset \{1,\ldots,u\}$. 
\end{enumerate}
\end{theorem} 
In particular if we use a $(q,q)$-Butson Hadamard matrix in the construction,  
we obtain the following. 
\begin{theorem}
For any prime power $q$, there exist $q$ mutually unbiased unit weighing matrices $W_1,\ldots,W_{q}$ for parameters $(q^2,q,q^2,1)$ such that $W_i W_j^*$ is a Bush-type $(q^2,q)$-Butson Hadamard matrix for any distinct $i,j\in\{1,\ldots,q\}$.
\end{theorem} 

\section*{Acknowledgement}
Hadi Kharaghani is supported by an NSERC Discovery Grant.  
Sho Suda is supported by JSPS KAKENHI Grant Number 15K21075. The authors wish to thank Darcy Best for showing the validity of Theorem~\ref{thm:maximal} for $t=1$ by computer computation.



\begin{thebibliography}{99}
\bibitem{A85}S.S. Agaian,
{\sl Hadamard Matrices and Their Applications},
Lecture Notes in Mathematics, 1168, 
Springer-Verlag, Berlin, 1985.

\bibitem{AHS}
M. Araya, M. Harada and S. Suda, 
Quasi-unbiased Hadamard matrices and weakly unbiased Hadamard matrices: a coding-theoretic approach, 
to appear in {\sl Math.\ of Comp}. 

\bibitem{BKR}
D. Best, H. Kharaghani and H. Ramp, 
Mutually unbiased weighing matrices, 
{\sl Des.\ Codes Cryptogr.}, {\bf 76} (2015), 237--256.

\bibitem{B}
R. C. Bose, 
Strongly regular graphs, partial geometries and partially balanced designs, 
{\sl Pacific J. Math.}, {\bf 13} (1963), 389--419. 

\bibitem{Bush}
K. A. Bush, 
Unbalanced Hadamard matrices and finite projective planes of even order. 
{\sl J. Combin. Theory Ser. A}, {\bf 11} (1971), 38--44. 

\bibitem{But}
A. T. Butson, 
Generalized Hadamard matrices. 
{\sl Proc. Amer. Math. Soc.}, {\bf 13} (1962), 894--898. 

\bibitem{dmm}
E. van Dam, W. Martin, M. Muzychuk, 
Uniformity in association schemes and coherent configurations: cometric Q-antipodal schemes and linked systems,
{\sl J. Combin. Theory Ser. A}, {\bf 120} (2013), no. 7, 1401--1439.

\bibitem{gs-72}
Anthony V. Geramita,  Jennifer Seberry,  Orthogonal designs. Quadratic forms and Hadamard matrices. Lecture Notes in Pure and Applied Mathematics, 45. Marcel Dekker, Inc., New York, 1979.

\bibitem{HT}
W. H. Haemers and V. D. Tonchev, 
Spreads in strongly regular graphs, 
{\sl Des. Codes and Crypt.}, {\bf 8} (1996),  145--157. 

\bibitem{hko}
W. H. Holzmann, H. Kharaghani, W. Orrick, 
On the real unbiased Hadamard matrices.
Combinatorics and graphs, 243--250, Contemp. Math., 531, Amer. Math. Soc.,
Providence, RI, 2010.

\bibitem{K}
H. Kharaghani, 
New class of weighing matrices, 
{\sl Ars. Combin.} {\bf 19} (1985), 69--72.

\bibitem{KSS}
H. Kharaghani, S. Sasani and S. Suda, 
Mutually unbiased Bush-type Hadamard matrices and association schemes, 
{\sl Elec.\ J.\ Combin.}, {\bf 22} (2015), P3.10.

\bibitem{LMO}
N. LeCompte, W.J. Martin and W. Owens, 
On the equivalence between real mutually unbiased bases and a certain class of association schemes, 
{\sl Eur.\ J.\ Combin.}, {\bf 31} (2010), 1499--1512.

\bibitem{NS}
H. Nozaki and S. Suda, 
Weighing matrices and spherical codes,
{\sl J. Alg.\ Combin.}, {\bf 42} (2015), 283--291.

\bibitem{R}
P. J. Robinson, 
Using product designs to construct orthogonal designs, 
{\sl Bull. Austral. Math. Soc.}, {\bf 16} (1977),  297--305. 

\bibitem{seb-yam}
 Jennifer Seberry, Mieko Yamada, Hadamard matrices, sequences, and block designs. Contemporary design theory, 431--560, Wiley-Intersci. Ser. Discrete Math. Optim., Wiley, New York, 1992.

\bibitem{js-asy}
 Wallis, Jennifer Seberry, On the existence of Hadamard matrices. {\sl J. Combin. Theory Ser. A},  {\bf 21} (1976), 188--195.

\bibitem{S}
S. Suda, 
A two-fold cover of strongly regular graphs with spreads and association schemes of class five,
{\sl Des.\ Codes Cryptogr.},  DOI 10.1007/s10623-014-0012-z.

\bibitem{W}
Z. Wan, 
{\sl Lectures on finite fields and Galois rings}, 
World Scientific Publishing Co., Inc., River Edge, NJ, 2003. x+342 pp.


\end{thebibliography}
\end{document}